\UseRawInputEncoding
\documentclass{amsart}
\usepackage[T1]{fontenc} 
\usepackage[english]{babel} 
\usepackage{amsmath,amsthm,graphicx,float,mathtools,tikz-cd,amssymb,hyperref} 
\usepackage[capitalize]{cleveref}
\newcommand{\mb}[1]{\mathbb{#1}}

\newcommand{\inj}{\hookrightarrow}

\newcommand{\gap}{\text{\space}}

\newcommand{\st}{\text{ : }}

\newcommand{\mc}[1]{\mathcal{#1}}

\newcommand{\g}{\gamma}

\newcommand{\p}{\rho}

\newcommand{\R}{\mathbb{R}}

\newcommand{\Z}{\mathbb{Z}}

\newcommand{\N}{\mathbb{N}}
\newcommand{\F}{\mathbb{F}}

\newcommand{\btz}{\begin{tikzcd}}
\newcommand{\etz}{\end{tikzcd}}
\newcommand{\x}[1]{\xrightarrow{#1}}

\newcommand{\bpm}{\begin{pmatrix}}
\newcommand{\epm}{\end{pmatrix}}
\newcommand{\bg}{\bigg}

\newcommand{\RP}{\mathbb{R}\mathbb{P}}

\newcommand{\smsh}{\wedge}
\newcommand{\tx}[1]{\text{#1}}

\newcommand{\ti}[1]{\textit{#1}}

\newcommand{\ov}[1]{\overline{#1}}

\newcommand{\aln}[1]{\begin{align*}#1\end{align*}}
\newcommand{\un}[1]{\underline{#1}}
\DeclareMathOperator{\colim}{\text{colim}}

\usepackage{graphicx}

\newtheorem{theorem}{Theorem}[section]
\newtheorem{lemma}[theorem]{Lemma}

\theoremstyle{definition}
\newtheorem{definition}[theorem]{Definition}

\theoremstyle{remark}
\newtheorem{remark}[theorem]{Remark}

\theoremstyle{construction}
\newtheorem{construction}[theorem]{Construction}

\theoremstyle{corollary}
\newtheorem{corollary}[theorem]{Corollary}

\theoremstyle{proposition}
\newtheorem{proposition}[theorem]{Proposition}

\numberwithin{equation}{section}

\begin{document}

\title{Cofreeness in Real Bordism Theory and the Segal Conjecture}

\author{Christian Carrick}

\address{University of California, Los Angeles, Los Angeles, CA 90095}
\email{carrick@math.ucla.edu}
\thanks{This material is based upon work supported by the National Science Foundation under Grant No. DMS-1811189}






\begin{abstract}
We prove that the genuine $C_{2^n}$-spectrum $N_{C_{2}}^{C_{2^n}}MU_{\mathbb{R}}$ is cofree, for all $n$. Our proof is a formal argument using chromatic hypercubes and the Slice Theorem of Hill, Hopkins, and Ravenel. We show that this gives a new proof of the Segal Conjecture for $C_2$, independent of Lin's theorem.
\end{abstract}

\maketitle

\section{Introduction}
In this paper, we establish the following result:
\begin{theorem}\label{1.1}
For all $n>0$, the $C_{2^n}$-spectrum $N_{C_{2}}^{C_{2^n}}MU_\R$ is cofree, i.e. the map
\[N_{C_{2}}^{C_{2^n}}MU_\R\to F(E{C_{2^n}}_+,N_{C_{2}}^{C_{2^n}}MU_\R)\]
is an equivalence.
\end{theorem}
The case $n=1$ of this result was proven by Hu and Kriz via direct computation \cite{Hu}. The equivariant spectra $N_{C_{2}}^{C_{2^n}}MU_\R$ play a central role in the solution to the Kervaire Invariant One problem by Hill, Hopkins, and Ravenel \cite{HHR}. Their detecting spectrum $\Omega$ is the homotopy fixed point spectrum of a localization $\Omega_{\mb O}:=D^{-1}N_{C_{2}}^{C_{8}}MU_\R$ of $N_{C_{2}}^{C_{8}}MU_\R$. An essential piece of their argument is the Homotopy Fixed Point Theorem (\cite{HHR}, 1.10), which states that this homotopy fixed point spectrum coincides with the \ti{genuine} fixed point spectrum, i.e. that $\Omega_{\mb O}$ is cofree. Our result shows that this holds even \ti{before} localization away from $D$.

We use \cref{1.1} to give a new, more conceptual proof of a fundamental and deep result in equivariant homotopy theory, the Segal Conjecture for $C_2$:
\begin{theorem}\label{1.2} For any bounded below spectrum $X$, the Tate diagonal
\[X\to (N_e^{C_2}(X))^{tC_2}\]
is a $2$-complete equivalence.
\end{theorem}
Originally proven by Lin \cite{Lin} in the case $X=S^0$, the Segal Conjecture for $C_2$ has been substantially generalized and was stated in the above form first by Lunøe-Nielsen and Rognes (\cite{LeNR}, 5.13), who established the theorem for $X$ with finitely generated homotopy groups. The result was extended to all $X$ bounded below by Nikolaus and Scholze (\cite{NS}, III.1.7). We refer the reader to the introduction of \cite{HW} for a discussion of the different forms of the Segal Conjecture for $C_2$ and their relation to Lin's Theorem.

Lin's proof involves a difficult calculation of a continuous Ext group
\[\widehat{\tx{Ext}}_{\mc A}(H^*(\RP^\infty_{-\infty};\F_2);\F_2)\]
where $\mc A$ is the Steenrod algebra. Nikolaus and Scholze showed, however, that \cref{1.2} follows formally for all $X$ bounded below from the case $X=H\F_2$. Hahn and Wilson \cite{HW} used this to show that \cref{1.2} can be established by analysis of the descent spectral sequence for the map
\[N_e^{C_2}H\F_2\to H\un{\F_2}\]
which reduces to a continuous Ext group calculation over a much smaller polynomial coalgebra $\F_2[x]$. 

We give a proof of Lin's Theorem that involves essentially no homological algebra and proceeds from a chromatic approach. Essential to our proof is the identification $\Phi^{C_2}(N_{C_2}^{C_4}BP_\R)\simeq N_e^{C_2}H\F_2$. In \cite{MSZ}, Meier, Shi, and Zeng use this identification to deduce differentials in the homotopy fixed point spectral sequence of $N_e^{C_2}H\F_2$ from differentials in the slice spectral sequence of $N_{C_2}^{C_4}BP_\R$, thus establishing a connection between the Segal Conjecture and the Hill-Hopkins-Ravenel (HHR) Slice Theorem. We make this connection precise by proving the following:

\begin{theorem}\label{1.3}
For any $n>1$, the cofreeness of $N_{C_2}^{C_{2^n}}MU_\R$ is equivalent to Lin's Theorem together with the cofreeness of $MU_\R$.
\end{theorem}

\cref{1.3} is a formal consequence of the Nikolaus-Scholze Tate orbit lemma (\cite{NS}, I.2.1), and this gives a straightforward proof of \cref{1.1} using Lin's Theorem and the result of Hu and Kriz. On the other hand, we give an independent proof of the cofreeness of $N_{C_2}^{C_{2^n}}MU_\R$ that works for \ti{all} $n>0$, following a chromatic approach which depends only on the HHR Slice Theorem (\cite{HHR}, 6.1).

We give a sketch here of this proof in the case $n=1$. The idea is that $BP_\R[\ov{v_i}^{-1}]$ is cofree for formal reasons, so one can take an approach via local cohomology and form cartesian cubes
\[
\btz
\tilde{L}_2BP_\R\arrow[r]\arrow[d]&BP_\R[\ov{v_1}^{-1}]\arrow[d]\\
BP_\R[\ov{v_2}^{-1}]\arrow[r]&BP_\R[(\ov{v_1}\ov{v_2})^{-1}]
\etz
\]
\[
\btz
&BP_\R[\ov{v_3}^{-1}]\arrow[rr]\arrow[dd]&&BP_\R[(\ov{v_2}\ov{v_3})^{-1}]\arrow[dd]\\
\tilde{L}_3BP_\R\arrow[ur]\arrow[dd]\arrow[rr,crossing over]&&BP_\R[\ov{v_2}^{-1}]\arrow[ur]\\
&BP_\R[(\ov{v_1}\ov{v_3})^{-1}]\arrow[rr]&&BP_\R[(\ov{v_1}\ov{v_2}\ov{v_3})^{-1}]\\
BP_\R[\ov{v_1}^{-1}]\arrow[ur]\arrow[rr]&&BP_\R[(\ov{v_1}\ov{v_2})^{-1}]\arrow[ur]\arrow[uu,leftarrow,crossing over]
\etz
\]
and so on, and $\tilde{L}_nBP_\R$ is cofree for all $n$. Applying the slice tower to each vertex $BP_\R[(\ov{v_{i_1}}\cdots\ov{v_{i_j}})^{-1}]$, one forms a cartesian cube in filtered $C_2$-spectra, and the limit term gives a modified slice filtration of $\tilde{L}_nBP_\R$. It is then a formal consequence of the HHR Slice Theorem that, taking the limit in $n$, one recovers the slice tower of $BP_\R$.


\begin{remark}
Our results should shed light on the spectral sequences studied in Meier, Shi, and Zeng \cite{MSZ}. In particular, the map from the Slice SS of $N_{C_2}^{C_4}BP_\R$ to its HFPSS is an isomorphism below a line of slope 1 (see \cite{Ullman}). The Slice SS vanishes above a line of slope 3, but there are many classes above this line in the HFPSS. By Theorem \hyperref[1.1]{1.1}, the map between them must give an isomorphism on their $E_\infty$-pages, so there must be some pattern of differentials killing all the classes above this line in the HFPSS.
\end{remark}

\subsection*{Summary} In \cref{2}, we show that the cofreeness of $N_{C_2}^{C_{2^n}}MU_\R$ follows formally from (and is equivalent to) the Hu-Kriz $n=1$ case together with Lin's Theorem. This is the most direct way to \cref{1.1}, using these known results. In \cref{3}, we withhold knowledge of these theorems and give a different proof - via chromatic hypercubes - that $N_{C_2}^{C_4}BP_\R$ is cofree. In turn, this result implies the $n=1$ case and Lin's Theorem, which then gives the result for $n>2$ by the same induction used in \cref{2}. 

\subsection*{Notation and Conventions} We use $Sp^G$ to denote the category of orthogonal $G$-spectra or the associated $\infty$-category given by taking the homotopy coherent nerve of bifibrant objects in the stable model structure of Mandell and May \cite{MM}. We use the notation $MU^{((G))}$ and $BP^{((G))}$ to denote $N_{C_2}^GMU_\R$ and $N_{C_2}^GBP_\R$ respectively, as in HHR.

\subsection*{Acknowledgments} The $n=1$ case of our chromatic hypercubes result - namely that $BP_\R=\mathrm{holim}_n\tilde{L}_nBP_\R$ - is due to Mike Hill. It was his idea to use this approach to establish the $n>1$ cases. We thank him for introducing us to this problem and for his guidance throughout the project. 


\section{Cofreeness and Gluing Maps}\label{2}
\subsection{Cofreeness} We begin by reviewing the notion of cofreeness for a genuine $G$-spectrum. 
\begin{proposition}\label{2.1}
For $X\in Sp^G$, the following are equivalent
\begin{enumerate}
\item $X\to F(EG_+,X)$ is an equivalence of $G$-spectra.
\item $X^H\to X^{hH}$ is an equivalence of spectra for all $H\subset G$.
\item $X$ is $G_+$-local.
\end{enumerate}
\end{proposition}
\begin{proof}
For $1\iff 3$, it suffices to show that $L_{G_+}(X)=F(EG_+,X)$. The map
\[X\to F(EG_+,X)\]
becomes an equivalence after smashing with $G_+$ by the Frobenius relation, and the target is $G_+$-local because if $Z\smsh G_+\simeq*$, then
\[[Z,F(EG_+,X)]^G=[Z\smsh EG_+,X]^G=0\]
as $EG_+$ is in the localizing subcategory generated by $G_+$. $1\iff 2$ follows from the fact that the fixed point functors $(-)^H$ are jointly conservative, and
\[i^G_H(F(EG_+,X))=F(EH_+,i^G_HX)\]
as can be seen from the more general statement
\[i^G_H(L_E(X))=L_{i^G_HE}(i^G_HX)\]
(see \cite{Smashing}, 3.2).
\end{proof}
\begin{definition}\label{2.2}
We say a $G$-spectrum $X$ is \ti{cofree} if any of the equivalent conditions in \cref{2.1} hold.
\end{definition}
\begin{corollary}\label{2.3}
The category of cofree $G$-spectra is closed under homotopy limits.
\end{corollary}
\begin{proof}
This is true of any category of $E$-locals.
\end{proof}
\begin{remark}
Cofree $G$-spectra are often called \ti{Borel complete}, or just \ti{Borel}. The source of this terminology is the fact that there is a forgetful functor
\[Sp^G\to \mathrm{Fun}(BG,Sp)\]
from genuine $G$-spectra to so-called Borel $G$-spectra. For formal reasons, this functor admits a right adjoint, and it is not hard to show that this right adjoint is an equivalence onto the full subcategory of cofree $G$-spectra.
\end{remark}

We will make use of the slice filtration on $G$-spectra, introduced for $C_2$-spectra by Dugger \cite{Dugger} and generalized to all finite groups $G$ by HHR \cite{HHR}. Let $X\ge n$ denote that a $G$-spectrum is slice $\ge n$, i.e. $X$ is slice $(n-1)$-connected. We need the following useful lemma:
\begin{lemma}\label{2.5}
Suppose $\{X_i\}_{i\in\N}$ is a family of $G$-spectra such that, for all $n\in\Z$, all but finitely many $X_i$ have the property that $X_i\ge n$. Then the canonical map
\[\bigvee\limits_iX_i\to\prod\limits_iX_i\]
is an equivalence.
\end{lemma}
\begin{proof}
It suffices to show that, for all $k\in\Z$, the map of Mackey functors 
\[\bigoplus\limits_i\un{\pi}_k(X_i)\cong\un{\pi}_k\bg(\bigvee\limits_iX_i\bg)\to\un{\pi}_k\bg(\prod\limits_iX_i\bg)\cong\prod\limits_i\un{\pi}_k(X_i)\]
is an isomorphism. This follows immediately from the observation that $\un{\pi}_k(X_i)=0$ for all but finitely many $i$. Indeed, by (\cite{HHR}, 4.40), if $Y\ge n$, then $\un{\pi}_k(Y)=0$ for $k<\lfloor n/|G|\rfloor$ when $n\ge0$ and for $k<n$ when $n\le 0$.
\end{proof}
\begin{proposition}\label{2.6}
If $MU_\R$ is cofree, then $MU_\R^{\smsh n}$ is cofree for all $n\ge1$, and similarly for $BP_\R^{\smsh n}$.
\end{proposition}
\begin{proof}
We proceed by induction on $n$. Since $MU_\R^{\smsh(n-1)}$ is Real-oriented, we have
\[MU_\R^{\smsh n}=MU_\R^{\smsh (n-1)}[\ov{b_1},\ov{b_2},\ldots]=\bigvee\limits_{m\in M}S^{\frac{|m|}{2}\rho}\smsh MU_\R^{\smsh(n-1)}\]
where $M$ is a monomial basis of $\Z[b_1,b_2,\ldots]$. By the lemma, the canonical map
\[\bigvee\limits_{m\in M}S^{\frac{|m|}{2}\rho}\smsh MU_\R^{\smsh(n-1)}\to\prod\limits_{m\in M}S^{\frac{|m|}{2}\rho}\smsh MU_\R^{\smsh(n-1)}\]
is an equivalence, as $MU_\R^{\smsh(n-1)}\ge0$ and $S^{k\rho}\ge 2k$, so that $S^{k\rho}\smsh MU_\R^{\smsh(n-1)}\ge2k$ by (\cite{HHR}, 4.26).

This completes the proof, as the category of cofree $C_{2}$-spectra is closed under limits and smashing with a dualizable $C_2$-spectrum, hence the target is cofree.
\end{proof}
\subsection{Gluing maps and cofreeness}
We set up an inductive argument to prove \cref{1.1}. To fix notation, we use $\Phi^{C_{p^k}}$ to denote the functor $Sp^{C_{p^n}}\to Sp$ and $\widetilde{\Phi}^{C_{p^k}}$ to denote the functor $Sp^{C_{p^n}}\to Sp^{C_{p^{n-k}}}$, so that $i^{C_{p^{n-k}}}_e\circ\widetilde{\Phi}^{C_{p^{k}}}=\Phi^{C_{p^k}}$. Nikolaus and Scholze use a result of Hesselholt and Madsen (\cite{HM}, 2.1) along with their Tate orbit lemma, to show $\mathrm{(}$\cite{NS}$\mathrm{,\gap}$$\mathrm{Corollary\gap II.4.7)}$:
\begin{proposition}\label{2.7}
If $X\in Sp^{C_{p^n}}$ has the property that $\Phi^{C_{p^k}}X\in Sp$ is bounded below for all $0\le k< n$, there is a homotopy limit diagram
\[
\btz
X^{C_{p^n}}\arrow[rrrr]\arrow[dddd]&&&&\Phi^{C_{p^n}}X\arrow[d]\\
&&&(\widetilde{\Phi}^{C_{p^{n-1}}}X)^{hC_{p}}\arrow[r]\arrow[d]&(\widetilde{\Phi}^{C_{p^{n-1}}}X)^{tC_{p}}\\
&&(\widetilde{\Phi}^{C_{p^2}}X)^{hC_{p^{n-2}}}\arrow[r]\arrow[d]&\cdots\\
&\bg(\widetilde{\Phi}^{C_{p}}X\bg)^{hC_{p^{n-1}}}\arrow[r]\arrow[d]&\bg((\widetilde{\Phi}^{C_{p}}X)^{tC_{p}}\bg)^{hC_{p^{n-2}}}\\
X^{hC_{p^n}}\arrow[r]&\bg(X^{tC_{p}}\bg)^{hC_{p^{n-1}}}
\etz
\]
\end{proposition}

\begin{theorem}\label{2.8} Let $Y$ be a bounded below $C_{p}$-spectrum.  If $Y^{\smsh p^k}$ is a cofree $C_p$-spectrum for all $0\le k<n$, then $N_{C_p}^{C_{p^n}}Y$ is cofree.
\end{theorem}
\begin{proof}
Set $X:=N_{C_p}^{C_{p^n}}Y$. We proceed by induction on $n$, with the base case $n=1$ being tautological. For all $1\le k<n$,
\[i^{C_{p^n}}_{C_{p^{n-k}}}X=N_{C_p}^{C_{p^{n-k}}}(Y^{\smsh p^{k}})\]
is cofree by induction, so it suffices to show the map $X^{C_{p^n}}\to X^{hC_{p^n}}$ is an equivalence. Since $Y$ is bounded below, so is $X$, and this map is an equivalence if all of the short vertical maps in \cref{2.7} are equivalences. Each such map is of the form
\[(f)^{hC_{p^{n-k}}}:\bg(\widetilde{\Phi}^{C_{p^k}}X\bg)^{hC_{p^{n-k}}}\to\bg((\widetilde{\Phi}^{C_{p^{k-1}}}X)^{tC_{p}}\bg)^{hC_{p^{n-k}}}\]
for $k>0$, which is induced by the map in $Sp^{C_{p^{n-k}}}$
\[f:\widetilde{\Phi}^{C_{p^k}}X\to(\widetilde{\Phi}^{C_{p^{k-1}}}X)^{tC_{p}}\]
It therefore suffices to show that $f$ is an equivalence of Borel $C_{p^{n-k}}$-spectra for all $k>0$, which by definition is simply an underlying equivalence. The underlying map is the natural map
\[\Phi^{C_p}\bg(i^{C_{p^{n-k+1}}}_{C_p}\tilde{\Phi}^{C_{p^{k-1}}}X\bg)\to \bg(i^{C_{p^{n-k+1}}}_{C_p}\tilde{\Phi}^{C_{p^{k-1}}}X\bg)^{tC_p}\]
so it suffices to show $i^{C_{p^{n-k+1}}}_{C_p}\tilde{\Phi}^{C_{p^{k-1}}}X$ is a cofree $C_p$-spectrum. When $k=1$, we have
\[i^{C_{p^{n-k+1}}}_{C_p}\tilde{\Phi}^{C_{p^{k-1}}}X\simeq Y^{\smsh p^{n-1}}\]
and for $k>1$, one has
\aln{i^{C_{p^{n-k+1}}}_{C_p}\tilde{\Phi}^{C_{p^{k-1}}}X\simeq i^{C_{p^{n-k+1}}}_{C_{p}}(N_e^{C_{p^{n-k+1}}}(\Phi^{C_p}Y))
\simeq N_e^{C_{p}}\bg(\Phi^{C_p}(Y^{\smsh p^{n-k}})\bg)}
using the identification $\widetilde{\Phi}^{C_{p^k}}X\simeq N_e^{C_{p^{n-k}}}(\Phi^{C_p}Y)$ (see \cite{MSZ}, Theorem 2.2). The $C_p$-spectrum $N_e^{C_{p}}\Phi^{C_p}(Y^{\smsh p^{n-k}})$ is cofree by the Segal Conjecture for $C_p$: since $Y^{\smsh p^{n-k}}$ is bounded below and cofree,
\[\Phi^{C_p}(Y^{\smsh p^{n-k}})\simeq (Y^{\smsh p^{n-k}})^{tC_p}\]
is bounded below and $p$-complete.
%
\end{proof}

\begin{remark}
This result has various converses. For example, if $Y$ is a bounded below $C_p$-spectrum, then $N_{C_p}^{C_{p^k}}Y$ is cofree for all $1\le k\le n$ if and only if $Y^{\smsh p^k}$ is a cofree $C_p$-spectrum for all $0\le k<n$. The other direction follows because if $N_{C_p}^{C_{p^{k+1}}}Y$ is cofree, then $Y^{\smsh p^k}=i^{C_{p^{k+1}}}_{C_p}N_{C_p}^{C_{p^{k+1}}}Y$ is also cofree.

If $Y$ is also a ring spectrum, then the direct converse of \cref{2.8} is true: $N_{C_p}^{C_{p^n}}Y$ is cofree if and only if $Y^{\smsh p^k}$ is a cofree $C_p$-spectrum for all $0\le k<n$. This follows because $Y^{\smsh p^k}$ is a retract of $Y^{\smsh p^{n-1}}=i^{C_{p^n}}_{C_p}N_{C_p}^{C_{p^n}}Y$ in this case.
\end{remark}

\begin{corollary}\label{2.10}
For all $n\ge1$, $MU^{((C_{2^n}))}$ is cofree, and similarly for $BP^{((C_{2^n}))}$.
\end{corollary}
\begin{proof}
$MU_\R$ is bounded below, so this follows immediately from \cref{2.6}, the Hu-Kriz $n=1$ case, and the theorem.
\end{proof}
We have shown that the case $n=1$, due to Hu and Kriz, along with Lin's Theorem, implies that $MU^{((C_{2^n}))}$ is cofree for all $n\ge1$. The argument can be reversed to point to another proof of Lin's Theorem, namely: 
\begin{proposition}\label{2.11}
For any $n>1$, the cofreeness of $MU^{((C_{2^n}))}$ implies both Lin's Theorem and the $n=1$ case.
\end{proposition}
\begin{proof}
If for any $n>1$, $MU^{((C_{2^n}))}$ is cofree, then a smash power of $BP^{((C_4))}$ is cofree, and it follows that $BP^{((C_4))}$ is cofree, as a retract; similarly for $BP_\R$ and therefore for its smash powers by \cref{2.6}. In this case, the limit diagram in \cref{2.7} is as follows:
\[
\btz
(BP^{((C_4))})^{C_4}\arrow[dd]\arrow[rr]&&\Phi^{C_4}(BP^{((C_4))})\arrow[d]\\
&(\widetilde{\Phi}^{C_{2}}BP^{((C_4))})^{hC_{2}}\arrow[r]\arrow[d]&(\widetilde{\Phi}^{C_{2}}BP^{((C_4))})^{tC_{2}}\\
(BP^{((C_4))})^{hC_4}\arrow[r]&((BP^{((C_4))})^{tC_{2}})^{hC_{2}}
\etz
\]
The lefthand vertical arrow is an equivalence by assumption, and the middle arrow is an equivalence since $BP_\R\smsh BP_\R$ is cofree. We find that the righthand vertical map is an equivalence, and this is the Tate diagonal $H\F_2\to (N_e^{C_{2}}H\F_2)^{tC_{2}}$, which is an equivalence if and only if Lin's Theorem holds, by (\cite{NS}, III.1.7).
\end{proof}

\section{Localizations of Norms of Real Bordism Theory}\label{3}
In this section, we give a proof of \cref{1.1} that is independent of both Lin's Theorem and the Hu-Kriz $n=1$ case. Our strategy is to show that $BP^{((C_4))}$ is cofree by mimicking the argument sketched in the introduction to show that $BP_\R$ is cofree. By \cref{2.11}, this implies Lin's Theorem as well as the $n=1$ case of \cref{1.1}, which gives the cases $n>2$ by \cref{2.8}.

To construct hypercubes analogous to those for $BP_\R$, we need a family of elements in $\pi_\bigstar^{C_4} BP^{((C_4))}$ to play the role of the $\ov{v_i}$'s, and we need $BP^{((C_4))}$ to become cofree upon inverting these elements. Following the discussion in (\cite{HHR}, Section 6), in $\pi_*^u(BP^{((C_4))})=\pi_*(BP\smsh BP)$, there are classes $\{t_i\}_{i\ge 1}$ with the property that
\[\pi_*^u(BP^{((C_4))})=\Z_{(2)}[t_i,\g(t_i)\st i\ge 1]\]
as a $C_4$-algebra, where $\g$ is the generator of $C_4$ and $\g^2(t_i)=-t_i$. The restriction map
\[\pi_{*\rho_2}^{C_2}(BP^{((C_4))})\to\pi_{2*}^u(BP^{((C_4))})\]
is an isomorphism. Lifting the classes $t_i$ along this map, we have classes 
\[\ov{t_i}\in\pi_{(2^i-1)\p_2}^{C_2}(BP^{((C_4))})\]
and using the $C_4$-commutative ring structure on $MU^{((C_4))}_{(2)}$, this gives classes 
\[N_{C_2}^{C_4}(\ov{t_i})\in\pi_{(2^i-1)\p_4}^{C_4}(BP^{((C_4))})\]

Inverting these classes, we may form hypercubes whose limits
\[\tilde{L}_nBP^{((C_4))}:=\mathrm{holim}_{\{i_1,\ldots,i_j\}\in\mc P_0([n])}(BP^{((C_4))}[N_{C_2}^{C_4}(\ov{t_{i_1}}\cdots\ov{t_{i_j}})^{-1}])\]
are easily shown to be cofree. It suffices then to establish that the natural map
\[BP^{((C_4))}\to\mathrm{holim}_n(\tilde{L}_nBP^{((C_4))})\]
is an equivalence. We determine the slice towers of each of the vertices
\[BP^{((C_4))}_{i_1,\ldots,i_j}:=BP^{((C_4))}[N_{C_2}^{C_4}(\ov{t_{i_1}}\cdots\ov{t_{i_j}})^{-1}]\]
which determines a filtration on $\tilde{L}_nBP^{((C_4))}$. We therefore analyze the above map in the category of \ti{filtered} $C_4$-spectra. We show that the associated graded of the filtration on $\tilde{L}_nBP^{((C_4))}$ splits as
\[(H\un{\Z}_{(2)}\smsh\widehat{W})\oplus\mc X_{n}\]
where $(H\un{\Z}_{(2)}\smsh\widehat{W})$ is the associated graded of the slice filtration on $BP^{((C_4))}$, and the map $\mc X_n\to\mc X_{n-1}$ is null, from which the result follows.

In \cref{sec3.1}, we begin with some general results on hypercubes that will allow us to deduce the associated graded of the filtration on $\tilde{L}_nBP^{((C_4))}$. In \cref{sec3.2}, we show that the functor sending a $G$-spectrum to its slice tower commutes with filtered colimits, allowing us to easily deduce the slice tower of $BP^{((C_4))}_{i_1,\ldots,i_j}$ from that of $BP^{((C_4))}$. We finish in \cref{sec3.3} by showing the results of \cref{sec3.1} apply - on associated graded - to the hypercubes discussed above, completing the proof.
\subsection{Generalities on Hypercubes}\label{sec3.1} We give some general results on hypercubes that look like (summands of) our chromatic hypercubes for $BP^{((C_4))}$, on associated graded. In this section, we use the language of $\infty$-categories following \cite{HTT}; in particular, we work in the model of quasicategories, and use stable $\infty$-categories following \cite{HA}. For a discussion of cubical diagrams in the context of $\infty$-categories, see (\cite{HA}, Section 6) or \cite{ACB}.

We fix $\mc C$ a stable $\infty$-category. Let $[n]$ denote the totally ordered set $\{1,\ldots,n\}$, and for $T$ a totally ordered set, let $\mc P(T)$ denote its power set regarded as a poset under inclusion. Let $\mc P_0(T)$ denote the sub-poset $\mc P(T)\setminus\{\emptyset\}$.


\begin{definition}\label{3.1}  An $n$-cube $\mc{X}$ in $\mc C$ is a functor $\mc{X}:\mc P([n])\to\mc C$, and a partial $n$-cube is a functor $\mc P_0([n])\to\mc C$. We say an $n$-cube $\mc X$ is \ti{cartesian} if the map
\[\mc X(\emptyset)\to\mathrm{holim}_{T\in\mc P_0([n])}\mc X(T)\]
is an equivalence.
\end{definition}

\begin{construction}\label{3.2}
Let $\mc P$ be a poset and objects $C_T\in\mc C$ for $T\in\mc P$ given. Regarding $C:T\mapsto C_T$ as a functor from the discrete category $\mathrm{ob}\mc P$, we obtain a diagram $\mc X_C:\mc P\to\mc C$ via left Kan extension along the inclusion $\mathrm{ob}\mc P\to\mc P$. Concretely, 
\[\mc X_C(T)=\bigoplus\limits_{S\le T}C_S\]
and the maps in $\mc X_C$ are the canonical inclusions.
\end{construction}
\begin{definition}\label{3.3}
When $\mc P=\mc P_0([n])$, we say a partial $n$-cube $\mc X:\mc P_0([n])\to\mc C$ is \ti{built from disjoint split inclusions} if $\mc X$ is equivalent to some $\mc X_C$ as in \cref{3.2}. If $\mc X$ is a cartesian $n$-cube such that the corresponding partial $n$-cube is built from disjoint split inclusions, we say $\mc X$ is a cartesian $n$-cube built from disjoint split inclusions.
\end{definition}

To make this definition clearer, note that any partial $2$-cube built from disjoint split inclusions is equivalent to one of the form
\[
\btz
&C_2\arrow[d]\\
C_1\arrow[r]&C_1\oplus C_2\oplus C_{12}
\etz
\]
and any partial $3$-cube built from disjoint split inclusions is equivalent to one of the form
\[
\btz
&C_{3}\arrow[rr]\arrow[dd]&&C_{2}\oplus C_{3}\oplus C_{23}\arrow[dd]\\
&&C_{2}\arrow[ur]\\
&C_1\oplus C_{3}\oplus C_{13}\arrow[rr]&&C_1\oplus C_{2}\oplus C_{3}\oplus C_{12}\oplus C_{13}\oplus C_{23}\oplus C_{123}\\
C_1\arrow[ur]\arrow[rr]&&C_1\oplus C_{2}\oplus C_{12}\arrow[ur]\arrow[uu,leftarrow,crossing over]
\etz
\]
where the inclusions are the canonical ones. We want to identify the limit of a diagram of this form, and we use a result of Antolín-Camarena and Barthel on computing limits of cubical diagrams inductively:
\begin{proposition}\label{3.4} $\mathrm{(}$\cite{ACB}$\mathrm{, 2.4)}$ Let $\mc X:\mc P_0([n])\to\mc C$ be a partial $n$-cube in $\mc C$. One has a pullback square
\[
\btz
\mathrm{holim}_{S\in\mc P_0([n])}\mc X(S)\arrow[r]\arrow[d]&\mathrm{holim}_{S\in\mc P_0([n-1])}\mc X(S)\arrow[d]\\
\mc X(\{n\})\arrow[r]&\mathrm{holim}_{S\in\mc P_0([n-1])}\mc X(S\cup\{n\})
\etz
\]
\end{proposition}
\begin{proposition}\label{3.5}
Let $\mc X$ be a partial $n$-cube in $\mc C$ built from disjoint split inclusions with respect to some choice of objects $\{C_T\}_{T\in\mc P_0([n])}$ as in \cref{3.2}. Then $\mc X$ satisfies
\begin{enumerate}
\item $\mathrm{holim}_{S\in\mc P_0([n])}\mc X(S)\simeq \Omega^{n-1}C_{\{1,\ldots,n\}}$
\item The map
\[\mathrm{holim}_{S\in\mc P_0([n])}\mc X(S)\to \mathrm{holim}_{S\in\mc P_0([n-1])}\mc X(S)\]
is nullhomotopic.
\end{enumerate}
\end{proposition}
\begin{proof} We proceed by induction on $n$. For $n=1$, a cartesian $1$-cube is an equivalence
\[\mathrm{holim}_{S\in\mc P_0([1])}\mc X(S)\x{\simeq}\mc X(\{1\})\]
and the map in (2) is the map to the terminal object. It is straightforward to show that the partial $(n-1)$-cube
\[\mc P_0([n-1])\to\mc P_0([n])\x{\mc X}\mc C\]
is built from disjoint split inclusions, and
\[\mc P_0([n-1])\x{-\cup\{n\}}\mc P_0([n])\x{\mc X}\mc C\]
is of the form $C_{\{n\}}\oplus\mc Z$ where $\mc Z$ is a partial $(n-1)$-cube built from disjoint split inclusions using the objects $\{C_T\oplus C_{T\cup\{n\}}\}_{T\in\mc P_0([n-1])}$, as in \cref{3.2}. By induction, \cref{3.4} gives a pullback square
\[
\btz
\mathrm{holim}_{S\in\mc P_0([n])}\mc X(S)\arrow[r]\arrow[d]&\Omega^{n-2}C_{\{1,\ldots,n-1\}}\arrow[d]\\
C_{\{n\}}\arrow[r]&C_{\{n\}}\oplus \Omega^{n-2}C_{\{1,\ldots,n-1\}}\oplus\Omega^{n-2}C_{\{1,\ldots,n\}}
\etz
\]
which is a cartesian $2$-cube built from disjoint split inclusions. It therefore suffices to prove the proposition in the case $n=2$, which is the claim that for objects $C_1,C_2,C_{12}\in\mc C$, there is a pullback square of the form
\[
\btz
\Omega C_{12}\arrow[r,"0"]\arrow[d,"0"']&C_2\arrow[d]\\
C_1\arrow[r]&C_1\oplus C_2\oplus C_{12}
\etz
\]
One may form a morphism of partial 2-cubes
\[
\btz
&*\arrow[d]\\
*\arrow[r]&C_{12}
\etz
\implies
\btz
&C_2\arrow[d]\\
C_1\arrow[r]&C_1\oplus C_2\oplus C_{12}
\etz
\]
via naturality of \cref{3.2} which, taking limits, constructs such a square. Taking fibers along the vertical maps, one has the identity map of $\Omega C_1\oplus\Omega C_{12}$; the square is therefore cartesian by (\cite{ACB}, 2.2).
\end{proof}

\subsection{Slice Towers and Chromatic Localizations}\label{sec3.2}
In this section, we use the slice filtration to work in the $\infty$-category $\mathrm{Fun}(\Z^{op},Sp^G)$ of filtered $G$-spectra (see \cite{HA}, 1.2.2). We refer to \cite{Dylan} for a treatment of the slice filtration in an $\infty$-categorical context. Let
\[\mc T:Sp^G\to\mathrm{Fun}(\Z^{op},Sp^G)\]
be the functor which associates to a $G$-spectrum its slice tower, which may be obtained as in (\cite{HA}, 1.2.1.17). We use the following notation in this context:
\begin{itemize}
\item $\widetilde{P}^k:\tx{Fun}(\Z^{op},Sp^{C_4})\x{\tx{ev}_k}Sp^{C_4}$
\item $\widetilde{P}^k_k=\tx{fib}(\widetilde{P}^k\to \widetilde{P}^{k-1})$
\item $P^k=\tilde{P}^k\circ\mc T$
\item $P^k_k=\tilde{P}^k_k\circ\mc T$
\item $\mathrm{holim}:\mathrm{Fun}(\Z^{op},Sp^{C_4})\to Sp^{C_4}$ is the functor sending a tower to its homotopy limit.
\end{itemize}
We use extensively that limits and colimits are computed pointwise in functor categories. We begin with a useful lemma:
\begin{lemma}\label{3.6}
The functor $\mc T:Sp^G\to\mathrm{Fun}(\Z^{op},Sp^G)$ commutes with filtered colimits.
\end{lemma}
\begin{proof}
Let 
\[X_1\to X_2\to X_3\to\cdots\]
be an ind-system of $G$-spectra with colimit $X$. For all $k\in\Z$, we have a map of cofiber sequences
\[
\btz
\colim_i P_{k+1}(X_i)\arrow[d]\arrow[r]&\colim_i X_i\arrow[d,"\simeq"]\arrow[r]&\colim_i P^k(X_i)\arrow[d]\\
P_{k+1}(X)\arrow[r]&X\arrow[r]&P^k(X)
\etz
\]
By the slice recognition principle (\cite{HHR}, 4.16), the left and right arrows are equivalences provided that $\colim_i P_{k+1}(X_i)$ is slice $>k$ and $\colim_i P^k(X_i)$ is slice $\le k$. The former follows from the fact that the subcategory
\[\tau_{>k}=\{Y\in Sp^G\st Y>k\}\]
is a localizing subcategory by definition, and the latter follows from the fact that slice spheres are compact.

The lemma now follows from the fact that equivalences in functor categories are detected pointwise.
\end{proof}

\cref{3.6} allows us to easily determine the slice tower of the $C_4$-spectrum $BP^{((C_4))}_{i_1,\ldots,i_j}$. We define $n$ so that $-n\rho_4=|N_{C_2}^{C_4}(\ov{t_{i_1}}\cdots\ov{t_{i_j}})^{-1}|$, and
\[BP^{((C_4))}_{i_1,\ldots,i_j}=\colim\bg(BP^{((C_4))}\x{ N_{C_2}^{C_4}(\ov{t_{i_1}}\cdots\ov{t_{i_j}})\cdot}\Sigma^{n\p_4}BP^{((C_4))}\x{N_{C_2}^{C_4}(\ov{t_{i_1}}\cdots\ov{t_{i_j}})\cdot}\cdots\bg)\]
\begin{proposition}\label{3.7}
There is an equivalence of filtered $C_4$-spectra
\[\mc T(BP^{((C_4))}_{i_1,\ldots,i_j})\simeq \mathrm{colim}_k\mc T(\Sigma^{kn\p_4}BP^{((C_4))})\]
In particular, the localization
\[BP^{((C_4))}\to BP^{((C_4))}_{i_1,\ldots,i_j}\]
induces the corresponding localization
\[H\un{\Z}_{(2)}\smsh S^0[C_4\cdot \ov{t_1},C_4\cdot\ov{t_2},\ldots]\to H\un{\Z}_{(2)}\smsh S^0[C_4\cdot \ov{t_1},C_4\cdot\ov{t_2},\ldots][C_4\cdot (\ov{t_{i_1}}\cdots \ov{t_{i_j}})^{-1}]\]
on slice associated-graded. The notation is as in \cite{HHR}, where 
\[S^0[C_4\cdot \ov{t_1},C_4\cdot\ov{t_2},\ldots][C_4\cdot (\ov{t_{i_1}}\cdots \ov{t_{i_j}})^{-1}]=N_{C_2}^{C_4}(S^0[\ov{t_1},\ov{t_2},\ldots][(\ov{t_{i_1}}\cdots\ov{t_{i_j}})^{-1}])\]
\end{proposition}
\begin{proof}
The first claim follows immediately from \cref{3.6}. The description of the slice associated graded of $BP^{((C_4))}_{i_1,\ldots,i_j}$ follows from the HHR Slice Theorem and (\cite{HHR}, Corollary 4.25), which implies that 
\[P^{l}_l(\Sigma^{kn\p_4}BP^{((C_4))})\simeq \Sigma^{kn\p_4}P^{l-4kn}_{l-4kn}BP^{((C_4))}\]
\end{proof}
\begin{remark}\label{3.8}
The previous proposition should be interpreted as follows: the slice tower for $BP^{((C_4))}_{i_1,\ldots,i_j}$ forgets to the ordinary Postnikov tower of $i^{C_4}_eBP^{((C_4))}_{i_1,\ldots,i_j}$, which has $P^{2d-1}_{2d-1}\simeq *$ and
\[P^{2d}_{2d}\simeq H\Z_{(2)}\smsh W_{2d}\]
where $W_{2d}$ is a wedge of $S^{2d}$'s over the set of monomials of degree $2d$ in
\[\pi_*^u(BP^{((C_4))}_{i_1,\ldots,i_j})=\Z_{(2)}[t_i,\g(t_i)\st i\ge 1][(t_{i_1}\cdots t_{i_j}\g(t_{i_1})\cdots\g(t_{i_j}))^{-1}]\]
The slice tower is an equivariant refinement of this wherein the odd slices vanish, $H\Z_{(2)}$ is replaced with $H\un{\Z}_{(2)}$, the spheres in $W_{2d}$ corresponding to a summand of the above $C_4$-module with stabilizer $C_2$ are grouped with their conjugates in a  
\[{C_4}_+\smsh_{C_2}S^{d\p_2},\]
the spheres corresponding to a $C_4$-fixed summand are replaced with $S^{\frac{d}{2}\p_4}$, and there are no free summands. We define $\widehat{W}^{i_1,\ldots,i_j}_{2d}$ and $\widehat{W}_{2d}$ by
\[P^{2d}_{2d}(BP^{((C_4))}_{i_1,\ldots,i_j})=H\un{\Z}_{(2)}\smsh\widehat{W}^{i_1,\ldots,i_j}_{2d}\]
\[P^{2d}_{2d}(BP^{((C_4))})=H\un{\Z}_{(2)}\smsh\widehat{W}_{2d}\]
\end{remark}

\subsection{Proof of Theorem 1.1} \label{sec3.3}
We introduce the chromatic $n$-cubes we need to prove \cref{1.1} and show they split as a summand that is constant in $n$ and a cartesian $n$-cube built from disjoint split inclusions. 
\begin{definition}\label{3.9}
Consider the following hypercubes:
\begin{enumerate}
\item Let $\mc H_n$ be the cartesian $n$-cube so that for $\{i_1,\ldots,i_j\}\in\mc P_0([n])$
\[\mc H_n(\{i_1,\ldots,i_j\})=BP^{((C_4))}_{i_1,\ldots,i_j}\]
One may form this cube inductively by working in the category of $MU^{((C_4))}_{(2)}$-modules and applying the functors $(-)[N(\ov{t_i})^{-1}]$. See (\cite{ACB}, 3.1) for a similar construction.
\item Let $\mc S_{n,d}$ be the cartesian $n$-cube defined on $\mc P_0([n])$ by 
\[\mc S_{n,d}:\mc P_0([n])\x{\mc H_n}Sp^{C_4}\x{P^{2d}_{2d}}Sp^{C_4}\]
\end{enumerate}
\end{definition}

With notation as in \cref{3.8}, we note that $\widehat{W}^{i_1,\ldots,i_j}_{2d}$ has $\widehat{W}_{2d}$ as a split summand for any $\{i_1,\ldots,i_j\}$, corresponding to the split inclusion
\[\pi_{2d}^u(BP^{((C_4))})\inj\pi_{2d}^u(BP^{((C_4))}_{i_1,\ldots,i_j})\]
This splitting is natural in $\{i_1,\ldots,i_j\}$, so we see that there is a splitting
\[\mc S_{n,d}\simeq (H\un{\Z}_{(2)}\smsh \widehat{W}_{2d})\oplus\mc X_{n,d}\]
where $\mc X_{n,d}$ is a cartesian $n$-cube satisfying
\[\mc X_{n,d}(\{i_1,\ldots,i_j\})=H\un{\Z}_{(2)}\smsh (\widehat{W}^{i_1,\ldots,i_j}_{2d}/\widehat{W}_{2d})\]
We have the following connection to the generalities in \cref{sec3.1}:
\begin{proposition}\label{3.10}
The cube $\mc X_{n,d}$ is a cartesian $n$-cube built from disjoint split inclusions.
\end{proposition}
\begin{proof}
$\mc X_{n,d}$ is cartesian by definition. The result - and the terminology - follows from the fact that for any $\{i_1,\ldots,i_j\}$, the maps
\[\pi_*^u(BP^{((C_4))}_{i_k})\inj\pi_*^u(BP^{((C_4))}_{i_1,\ldots,i_j})\]
are split inclusions, and after factoring out $\pi_*^u(BP^{((C_4))})$, the maps
\[\iota_k:\frac{\pi_*^u(BP^{((C_4))}_{i_k})}{\pi_*^u(BP^{((C_4))})}\inj\frac{\pi_*^u(BP^{((C_4))}_{i_1,\ldots,i_j})}{\pi_*^u(BP^{((C_4))})}\]
are split inclusions with the property that $\mathrm{im}(\iota_k)\cap\mathrm{im}(\iota_{k'})=\{0\}$ for $k\neq k'$. 
Now the claim follows from the fact that
\[\frac{\pi_*^u(BP^{((C_4))}_{i_1,\ldots,i_j})}{\pi_*^u(BP^{((C_4))})}=\bg(\bigoplus\limits_{\substack{T<\{i_1,\ldots,i_j\}\\T\in\mc P_0(n)}}\frac{\pi_*^u(BP^{((C_4))}_{T})}{\pi_*^u(BP^{((C_4))})}\bg)\oplus\frac{(t_{i_1}\cdots t_{i_j}\g(t_{i_1})\cdots\g(t_{i_j}))^{<0}\pi_*^u(BP^{((C_4))})}{\pi_*^u(BP^{((C_4))})}\]
where the latter summand denotes the subgroup of $\frac{\pi_*^u(BP^{((C_4))}_{i_1,\ldots,i_j})}{\pi_*^u(BP^{((C_4))})}$ generated by monomials containing $(t_{i_1}\cdots t_{i_j}\g(t_{i_1})\cdots\g(t_{i_j}))^{-k}$ for $k>0$.
\end{proof}
The following is an immediate consequence of \cref{3.5} and \cref{3.10}:
\begin{corollary}\label{3.11}
The map $\mc S_{n,d}(\emptyset)\to\mc S_{n-1,d}(\emptyset)$ can be identified with
\[(H\un{\Z}_{(2)}\smsh\widehat{W}_{2d})\oplus\mc X_{n,d}(\emptyset)\x{\bpm 1&0\\
0&0\epm} (H\un{\Z}_{(2)}\smsh\widehat{W}_{2d})\oplus\mc X_{n-1,d}(\emptyset)\]
\end{corollary}
The canonical map $BP^{((C_4))}\to BP^{((C_4))}_{i_1,\ldots,i_j}$, by universal property,  determines compatible maps $BP^{((C_4))}\to \mc H_{n}(\emptyset)$ so that there is a map
\[BP^{((C_4))}\to\mathrm{holim}_n\mc H_n(\emptyset)\]
We will show this map is an equivalence, and this will complete the proof that $BP^{((C_4))}$ is cofree by the following:
\begin{proposition}\label{3.12} The $C_4$-spectrum $\mathrm{holim}_n\mc H_n(\emptyset)$ is cofree.
\end{proposition}
\begin{proof}
By \cref{2.3}, the category of cofree $C_4$-spectra is closed under limits, hence it suffices to show that each $\mc H_n(\emptyset)$ is cofree. There is by definition an equivalence
\[\mc H_n(\emptyset)\x{\simeq}\mathrm{holim}_{T\in\mc P_0([n])}\mc H_n(T)=\mathrm{holim}_{\{i_1,\ldots,i_j\}\in\mc P_0([n])}BP^{((C_4))}_{i_1,\ldots,i_j}\]
so it suffices to show that each $BP^{((C_4))}_{i_1,\ldots,i_j}$ is cofree. $BP^{((C_4))}_{i_1,\ldots,i_j}$ is a module over $MU^{((C_4))}_{(2)}[N_{C_2}^{C_4}(\ov{t_{i_1}}\cdots\ov{t_{i_j}})^{-1}]$ and so we may argue as in (\cite{HHR}, Section 10): we have that 
\[\Phi^{C_4}(MU^{((C_4))}_{(2)}[N_{C_2}^{C_4}(\ov{t_{i_1}}\cdots\ov{t_{i_j}})^{-1}])\simeq \Phi^{C_2}(MU^{((C_4))}_{(2)}[N_{C_2}^{C_4}(\ov{t_{i_1}}\cdots\ov{t_{i_j}})^{-1}])\simeq*\]
as
\[\Phi^{C_4}(N_{C_2}^{C_4}(\ov{t_{i_1}}))=\Phi^{C_2}(\ov{t_{i_1}})=0\]
and similarly
\[\Phi^{C_2}(N_{C_2}^{C_4}(\ov{t_{i_1}}))=\Phi^{C_2}(i^{C_4}_{C_2}N_{C_2}^{C_4}(\ov{t_{i_1}}))=\Phi^{C_2}(\ov{t_{i_1}}\cdot\ov{\g(t_{i_1})})=0\]
\end{proof}
To show that the map 
\[BP^{((C_4))}\to\mathrm{holim}_n\mc H_n(\emptyset)\]
is an equivalence, we work instead in filtered $C_4$-spectra, where by functoriality we have a map
\[f:\mc T(BP^{((C_4))})\to\mathrm{holim}_n\bg(\mathrm{holim}_{\{i_1,\ldots,i_j\}\in\mc P_0([n])}\mc T(BP^{((C_4))}_{i_1,\ldots,i_j})\bg)\]
We will show that $f$ is an equivalence, for which we need the following lemma:
\begin{lemma}\label{3.13}
Let $\mc C$ be a co-complete stable $\infty$-category. Suppose $\mc T_1,\mc T_2\in \mathrm{Fun}(\Z^{op},\mc C)$ are such that
\[\mathrm{colim}_k\tilde{P}^k(\mc T_1)\simeq \mathrm{colim}_k\tilde{P}^k(\mc T_2)\simeq*\]
If $\phi:\mc T_1\to\mc T_2$ has the property that $\tilde{P}^k_k(\phi)$ is an equivalence for all $k\in\Z$, then $\phi$ is an equivalence.
\end{lemma}
\begin{proof}
Let $\mc T_3=\mathrm{cofib}(\phi:\mc T_1\to\mc T_2)$, then it suffices to show that $\mc T_3\simeq *$. We have that $\tilde{P}^k_k(\mc T_3)\simeq*$ for all $k\in\Z$ so that 
\[\tilde{P}^k(\mc T_3)\to\tilde{P}^{k-1}(\mc T_3)\]
is an equivalence for all $k\in\Z$. Therefore $\mc T_3$ is equivalent to a constant tower, but since $\colim_k\tilde{P}^k(\mc T_3)\simeq*$, we must have $\mc T_3\simeq *$.
\end{proof}
\begin{theorem}\label{3.14}
The $C_4$-spectrum $BP^{((C_4))}$ is cofree, independent of Lin's Theorem.
\end{theorem}
\begin{proof}
It suffices to show that $f$ is an equivalence, as
\[\mathrm{holim}\bg(\mathrm{holim}_{\{i_1,\ldots,i_j\}\in\mc P_0([n])}\mc T(BP^{((C_4))}_{i_1,\ldots,i_j})
\bg)\simeq\mc H_n(\emptyset)\]
and 
\[\mathrm{holim}(\mc T(BP^{((C_4))}))\simeq BP^{((C_4))}\]
Note that 
\aln{
\widetilde{P}^k_k\bg(\mathrm{holim}_{\{i_1,\ldots,i_j\}\in\mc P_0([n])}\mc T(BP^{((C_4))}_{i_1,\ldots,i_j})\bg)&\simeq \mathrm{holim}_{\{i_1,\ldots,i_j\}\in\mc P_0([n])}\widetilde{P}^k_k\bg(\mc T(BP^{((C_4))}_{i_1,\ldots,i_j})\bg)\\
&\simeq\begin{cases}*&k=2d-1\\\mc S_{n,d}(\emptyset)&k=2d\end{cases}
}
The map $\widetilde{P}^{2d}_{2d}(f)$ is then identified with the map
\[H\un{\Z}_{(2)}\smsh\widehat{W}_{2d}\to\mathrm{holim}_n((H\un{\Z}_{(2)}\smsh\widehat{W}_{2d})\oplus \mc X_{n,d})\simeq \mathrm{holim}_n(H\un{\Z}_{(2)}\smsh\widehat{W}_{2d})\oplus\mathrm{holim}_n\mc X_{n,d}(\emptyset)  \]
By \cref{3.11}, the lefthand summand is constant in $n$, and the righthand summand is pro-zero, hence the map is an equivalence.

To establish that $f$ is an equivalence, by \cref{3.13}, it suffices now to show that 
\[\mathrm{colim}_k\widetilde{P}^{k}\bg(\mathrm{holim}_n\bg(\mathrm{holim}_{\{i_1,\ldots,i_j\}\in\mc P_0([n])}\mc T(BP^{((C_4))}_{i_1,\ldots,i_j})\bg)\bg)\simeq *\]
i.e. that the filtration on $\mathrm{holim}_n\mc H_n(\emptyset)$ \ti{strongly} converges. Note that by (\cite{HHR}, 4.42), if $X\in Sp^{C_4}$, then $\un{\pi}_{l}(P^{k}X)=0$ for $l\ge \lfloor (k+1)/4\rfloor$ when $k<0$ and for $k>l$ when $k\ge0$. Taking limits, it follows that
\[\un{\pi}_{l}\bg(\widetilde{P}^{k}\bg(\mathrm{holim}_{\{i_1,\ldots,i_j\}\in\mc P_0([n])}\mc T(BP^{((C_4))}_{i_1,\ldots,i_j})\bg)\bg)=0\]
in the same range, and so
\[\un{\pi}_{l}\bg(\widetilde{P}^{k}\bg(\mathrm{holim}_n\bg(\mathrm{holim}_{\{i_1,\ldots,i_j\}\in\mc P_0([n])}\mc T(BP^{((C_4))}_{i_1,\ldots,i_j})\bg)\bg)\bg)=0\]
in the same range by the Milnor sequence. It follows that, for any $l$, taking the colimit as $k\to-\infty$ of $\un{\pi}_l$ gives zero.
\end{proof}

\begin{remark}
This result recovers the Hu-Kriz result that $BP_\R$ is cofree: since $BP^{((C_4))}$ is cofree, $i^{C_4}_{C_2}BP^{((C_4))}=BP_\R\smsh BP_\R$ is cofree, hence so is the retract $BP_\R$. Alternatively, as discussed in the introduction, one may argue similarly to \cref{3.14} to show that $BP_\R$ is cofree, and the result in this case is due to Mike Hill. 
\end{remark}

\bibliographystyle{plain}

\end{document}